\documentclass[10pt,reqno]{amsart}
\usepackage{amsmath,amssymb,amsthm}
\usepackage{microtype} 
\usepackage{cite}
\usepackage{enumitem, xcolor}
\setlist[enumerate,1]{label=(\alph*), itemsep=5pt+2pt, leftmargin=*}
\usepackage{hyperref}
\usepackage{mathtools}
\mathtoolsset{showonlyrefs}

\allowdisplaybreaks

\newtheorem{theorem}{Theorem}[section]

\newtheorem{lemma}[theorem]{Lemma}
\newtheorem{corollary}[theorem]{Corollary}

\theoremstyle{definition}

\newtheorem{remark}[theorem]{Remark}
\newtheorem{example}[theorem]{Example}

\newcommand{\C}{\mathbb{C}}
\newcommand{\D}{\mathbb{D}}

\newcommand{\De}{\mathbb{D}_{\mathrm{e}}}
\newcommand{\K}{\mathcal{K}}
\renewcommand{\P}{\mathbb{CP}}
\newcommand{\R}{\mathbb{R}}
\newcommand{\T}{\mathbb{T}}

\newcommand{\caplog}{\operatorname{cap}}

\newcommand{\Span}{\operatorname{span}}
\newcommand{\ord}{\operatorname{ord}}
\newcommand{\capc}{\operatorname{cap}}

\renewcommand{\Im}{\operatorname{Im}}

\renewcommand{\phi}{\varphi}

\newcommand{\PP}{\P}
\newcommand{\Cp}{\C_+}
\newcommand{\Cm}{\C_-}
\newcommand{\N}{\mathcal{N}}

\title[Fermat's Last Theorem in star-invariant subspaces {II}]{Fermat's Last Theorem in star-invariant subspaces {II}:  small boundary spectrum, projective plane curves, and the cubic Fermat equation}
\author{Stephan Ramon Garcia}
\address{Department of Mathematics and Statistics, Pomona College, 610 N. College Ave., Claremont, CA 91711, USA}
\email{stephan.garcia@pomona.edu}
\urladdr{\url{https://stephangarcia.sites.pomona.edu/}}

\thanks{Partially supported by NSF grant DMS-2452084.}


\begin{document}

\begin{abstract}
We study algebraic relations among functions in star-invariant subspaces.  Let $\theta$ be an inner function and suppose that $f,g,h\in K_{\theta}^p$ satisfy $P(f,g,h)=0$, in which $P$ is a homogeneous polynomial and the associated projective plane curve is smooth.  Then $\dim\Span\{f,g,h\}\leq1$ if either $\deg P\geq4$ and the boundary spectrum of $\theta$ has logarithmic capacity zero, or $\deg P=3$ and the boundary spectrum contains at most one point.  This settles new cases of a problem of Dyakonov, subsumes the previously known results, and recovers the classical fact that polynomial parametrizations of a smooth projective plane curve of degree at least three are projectively constant.
\end{abstract}

\maketitle

\section{Introduction}

For $1 \leq p \leq \infty$, let $H^p$ denote the Hardy space on the open unit disk $\D$ \cite{Duren, Koosis, Nikolski, Garnett}.
Every function in $H^p$ has nontangential boundary values almost everywhere on the unit
circle $\T$, and its boundary function belongs to $L^p(\T)$. We identify Hardy-space functions with their boundary 
values without further comment.  

For each inner function $\theta$ on $\D$, we have a \emph{star-invariant subspace} $K_{\theta}^p = H^p \cap \theta \overline{zH^p}$.
These subspaces are invariant for the backward shift on $H^p$ \cite{CimaRoss, DSS, CBSTK}.  For $p=2$, we obtain the \emph{model space} $K_{\theta}^2 = H^2 \ominus \theta H^2$ from the study of Hilbert-space contractions \cite{IMSTO, MS, MSAS, Nikolski}.
The \emph{spectrum} $\sigma(\theta)$ of $\theta$  is the closure of the union of the zero set of its Blaschke factor
and the support of the singular measure from its singular-inner factor \cite[Def.~7.3.2]{IMSTO}.
The set of points of $\T$ across which $\theta$ cannot be analytically continued is the \emph{boundary spectrum} 
$\sigma(\theta)\cap \T$ \cite[Thm.~7.3.4.i \& Prop.~7.3.8]{IMSTO}.  In what follows, 
$\#(\sigma(\theta)\cap \T)$ and $\caplog(\sigma(\theta)\cap \T)$ denote 
the cardinality and logarithmic capacity of the boundary spectrum, respectively.

In 2010, Dyakonov posed the question of proving ``Fermat's Last Theorem'' for star-invariant subspaces \cite[Prob.~2]{Dyakonov}.
A triple $(f,g,h)$ of functions in $K_{\theta}^p$ is \emph{trivial} if $\dim\Span\{f,g,h\}\leq1$ and it is \emph{nontrivial} otherwise.
Dyakonov asked whether $f^n+g^n=h^n$ with $f,g,h \in K_{\theta}^p$ and $n \in \{3,4,5,\ldots\}$
implies that $(f,g,h)$ is trivial.  If $\theta$ is a finite Blaschke product, Dyakonov noted that the desired result follows from
the Mason--Stothers theorem \cite{Mason,Stothers}, a polynomial analogue of the $abc$-conjecture.

The author proved Dyakonov's conjecture when $n \geq 4$ and the boundary spectrum $\sigma(\theta) \cap \T$ has logarithmic capacity zero \cite[Thm.~1.1]{GarciaFLT}.  In this note we extend and generalize this result.  First, we consider certain equations 
$P(x,y,z)=0$ in which $P(x,y,z) \in \C[x,y,z]$ is homogeneous; Fermat's Last Theorem concerns
$P(x,y,z) = x^n+y^n-z^n$.  Second, we treat the previously untouched cubic case $n=3$ when
$\sigma(\theta)\cap \T$ contains at most one point.  For example, this applies when $\theta$ is the product of an atomic singular
inner function and a Blaschke product whose zeros accumulate only at the atom.
Our main result is the following theorem.

\begin{theorem}\label{Theorem:Main}
Let $1\leq p\leq\infty$, let $\theta$ be an inner function on $\D$, and
let $P(x,y,z) \in \C[x,y,z]$ be homogeneous of degree $n \geq 3$.  Suppose that the projective plane curve
\begin{equation}\label{eq:Gamma}
\Gamma_P = \big\{ [x:y:z]\in\PP^2 : P(x,y,z)=0 \big\}
\end{equation}
is smooth and $f,g,h\in K_{\theta}^p$ satisfy $P(f,g,h)=0$.  Then $\dim\Span\{f,g,h\}\leq 1$
if either of the following holds:
\begin{enumerate}[leftmargin=*]
\item $n=3$ and $\#(\sigma(\theta)\cap \T) \leq 1$; or
\item $n \geq 4$ and $\caplog(\sigma(\theta)\cap \T) = 0$.
\end{enumerate}
\end{theorem}

We remind the reader that two-dimensional complex projective space is $\P^2 = (\C^3 \setminus \{\mathbf{0}\}) / \sim$, in which $\sim$ identifies nonzero scalar multiples. We denote the equivalence class of $(x,y,z)$ by $[x : y : z]$.  With its standard affine charts $\P^2$ is a compact complex manifold of complex dimension $2$ \cite[Ch.~I, Sec.~3]{Miranda}.
The projective plane curve $\Gamma_P$ is \emph{smooth} if and only if $P$ and its partial derivatives
have no common zero in $\PP^2$.  
No irreducibility hypothesis is needed in Theorem~\ref{Theorem:Main} since a smooth projective plane curve is automatically irreducible \cite[p.~15]{Miranda}.

Theorem~\ref{Theorem:Main} shows that ``Fermat's Last Theorem for star-invariant subspaces'' 
holds under (a) or (b). More generally, it provides a unified treatment of homogeneous polynomial 
equations whose associated projective plane curves are smooth.

\begin{example}\label{Example:GenusOne}
If $\#(\sigma(\theta)\cap \T) \leq 1$ and $1\leq p \leq \infty$, then there are no nontrivial 
triples $(f,g,h)$ in $\K_{\theta}^p$ such that $P(f,g,h) = 0$, where $P(x,y,z)$ is any of the following:
\begin{align*}
P(x,y,z) &= \text{$x^3+y^3- z^3$} && (\text{Fermat cubic}),\\
P(x,y,z) &= \text{$ax^3+by^3+cz^3$ with $abc \neq 0$} && (\text{Diagonal cubic}),\\
P(x,y,z) &= \text{$x^3+y^3+z^3-3\lambda xyz$ with $\lambda^3 \neq 1$} && (\text{Hesse normal form}),\\
P(x,y,z) &= \text{$y^2z-x^3-Axz^2-Bz^3$ with $4A^3+27B^2\neq0$} && (\text{Weierstrass form}),\\
P(x,y,z) &= \text{$y^2z-x(x-z)(x-\lambda z)$ with $\lambda \neq 0,1$} && (\text{Legendre form}).
\end{align*}
Indeed, each of these projective plane curves is smooth, so Theorem~\ref{Theorem:Main}.a applies.
\end{example}

\begin{example}\label{Example:GenusHigher}
If $\capc(\sigma(\theta)\cap\T) = 0$ and $1\leq p \leq \infty$, then there are no nontrivial triples
$f,g,h$ in $\K_{\theta}^p$ such that $P(f,g,h) = 0$, where $P(x,y,z)$ is any of the following: 
\begin{align*}
P(x,y,z) &= \text{$x^n+y^n- z^n$ with $n\geq 4$} && (\text{Fermat curve}),\\
P(x,y,z) &= \text{$ax^n+by^n+cz^n$ with $abc \neq 0$ and $n\geq 4$} && (\text{Diagonal curve}),\\
P(x,y,z) &= \text{$xy^{n-1}+yz^{n-1}+zx^{n-1}$ with $n \geq 4$} && (\text{Klein curve}),\\
P(x,y,z) &= \text{$y^{n-1}z-x^n+z^n$ with $n \geq 4$} && (\text{Catalan curve}),\\
P(x,y,z) &= \text{\small$x^4+y^4+z^4-\lambda(x^2y^2+y^2z^2+x^2z^2)$ with $\lambda\notin\{ -2,1,2\}$}&& (\text{Ciani quartic}),\\
P(x,y,z) &= 27x^6+9xz^5+9xy^5-135x^4yz-45x^2y^2z^2+10y^3z^3 && (\text{Wiman sextic}),\\
P(x,y,z) &=144(x^4+y^4)+350x^2y^2-225(x^2+y^2)z^2+81z^4 && (\text{Trott curve}).
\end{align*}
Indeed, each of these projective plane curves is smooth, so Theorem~\ref{Theorem:Main}.b applies.
\end{example}

Theorem~\ref{Theorem:Main} also recovers a classical result from algebraic geometry
which follows from the Riemann--Hurwitz formula \cite[Thm.~II.4.16]{Miranda}.

\begin{corollary}\label{Corollary:Polynomial}
Let $P(x,y,z)$ be a homogeneous polynomial of degree $n\geq3$ such that $\Gamma_P\subset\PP^2$ is smooth.  If
$f,g,h$ are polynomials such that $P(f,g,h) = 0$,  then $\dim \Span\{f,g,h\}\leq 1$.
\end{corollary}

\begin{proof}
Suppose that $f,g,h \in K_{z^N}^2 = \Span\{1,z,\ldots,z^{N-1}\}$ and $P(f,g,h) = 0$.
Since $\sigma(z^N)\cap\T=\varnothing$, Theorem~\ref{Theorem:Main} ensures that
$\dim\Span\{f,g,h\} \leq 1$. 
\end{proof}

\begin{remark}\label{Remark:Weierstrass}
The distinction between $n=3$ and $n\geq4$ in Theorem~\ref{Theorem:Main} is natural.
A smooth projective plane curve $\Gamma_P$ of degree $n$ has genus
\begin{equation}\label{eq:Genus}
g(\Gamma_P)=\frac{(n-1)(n-2)}2;
\end{equation}
see \cite[Ch.~III, Prop.~2.6]{Miranda}.  Its universal cover is $\C$ if $g(\Gamma_P)=1$ and $\D$ if
$g(\Gamma_P)\geq2$  \cite[p.~443]{Gamelin}, \cite[Thm.~4.4.1 \& Cor.~4.4.1]{Jost}.  
If $n\geq4$, then $g(\Gamma_P)\geq3$, so its universal cover is $\D$ and hence every holomorphic map from $\C$ to $\Gamma_P$ lifts to a  holomorphic map from $\C$ to $\D$, which is constant by Liouville.
For $n=3$, however, $\Gamma_P$ has genus $1$ and universal cover $\C$.
Nonconstant meromorphic solutions can appear.
Let $\wp$ be the Weierstrass elliptic function such that
$(\wp')^2=4\wp^3-1$ \cite[Ch.~7, Sec.~3.3]{Ahlfors}.  Then
\begin{equation*}
f=1+\frac{\wp'}{\sqrt3},
\qquad
g=e^{2\pi i/3}\left(1-\frac{\wp'}{\sqrt3}\right),
\quad\text{and}\quad
h=2\wp
\end{equation*}
satisfy $f^3+g^3=h^3$ \cite{Baker,GrossBAMS,GrossMonthly}. 
Restricting to a pole-free disk and reparametrizing by $\D$ gives a nonconstant holomorphic solution on $\D$.
Thus, hyperbolic uniformization assists in the case $n \geq 4$, while the cubic case requires additional information.
This distinction is reflected in the two parts of Theorem~\ref{Theorem:Main}.
\end{remark}

This paper is structured as follows.  Section \ref{Section:Preliminaries} contains preliminary material common to both parts
of Theorem \ref{Theorem:Main}.  The proofs of parts (a) and (b) of Theorem \ref{Theorem:Main}
are in Sections \ref{Section:ProofA} and \ref{Section:ProofB}, respectively.  

\smallskip\noindent\textbf{Acknowledgments.}
The author acknowledges fruitful conversations with ChatGPT 5.6.  However, any errors or omissions are the fault of the author.

\section{Preliminaries}\label{Section:Preliminaries}
This section contains material needed for the proofs of both parts of Theorem~\ref{Theorem:Main}.
Material specific to parts (a) and (b)  appears in Sections \ref{Section:ProofA} and \ref{Section:ProofB}, respectively.  

\subsection{Conjugation and continuation}\label{Subsection:Pseudo}
Recall that $f \mapsto \overline{fz}\theta$ is a conjugation on $K_{\theta}^p$:
it is conjugate linear, isometric, and involutive \cite{CBSTK,CCO,IMSTO}.
For each $f \in K_{\theta}^p$, there is a unique $\widetilde{f} \in K_{\theta}^p$ such that
\begin{equation}\label{eq:Conjugate}
f = \theta \overline{z \widetilde{f}}
\end{equation}
almost everywhere on $\T$. If $F$ is holomorphic on $\D$, then its reflection 
\begin{equation}\label{eq:Reflection}
F^{\#}(z)=\overline{F(1/\overline z)}
\end{equation}
through $\T$ is holomorphic on $\De=\{z\in\widehat{\C}:|z|>1\}$, in which $\widehat{\C} = \C \cup \{ \infty\}$.

If $I \subseteq \T \setminus (\sigma(\theta)\cap\T)$ is an open arc, then $\theta$ and every member of $K_{\theta}^p$ continue analytically across $I$ \cite[Thm.~7.3.4.i \& Prop.~7.3.8]{IMSTO}.  On a neighborhood of $I$,
\begin{equation}\label{eq:PseudocontinuationLocal}
f = \frac{\theta}{z} \widetilde{f}^{\#}, \qquad
g = \frac{\theta}{z} \widetilde{g}^{\#} , \quad \text{and} \quad
h= \frac{\theta}{z} \widetilde{h}^{\#}
\end{equation}
for all $f,g,h \in K_{\theta}^p$.  Wherever the projective triples are defined,
\begin{equation*}
[f:g:h] = [\widetilde{f}^{\#}:\widetilde{g}^{\#}:\widetilde{h}^{\#}];
\end{equation*}
that is, the common factor $\theta/z$ disappears upon projectivization.

Suppose that $P$ is homogeneous of degree $n$ and
$P(f,g,h)=0$ on $\D$.  This identity extends across $I$, and on the exterior side of the
preceding neighborhood we have
\begin{equation}
0 =P(f,g,h) =P\bigg(\frac{\theta}{z}\widetilde{f}^{\#}, \frac{\theta}{z}\widetilde{g}^{\#}, \frac{\theta}{z}\widetilde{h}^{\#}\bigg)
=\bigg(\frac{\theta}{z}\bigg)^n P\big(\widetilde{f}^{\#},\widetilde{g}^{\#},\widetilde{h}^{\#} \big).
\end{equation}
Since $\theta/z$ is nonvanishing there, we have
\begin{equation}\label{eq:ExteriorPolynomialRelation}
P(\widetilde{f}^{\#},\widetilde{g}^{\#},\widetilde{h}^{\#})=0
\end{equation}
on $\De$ by the identity theorem.

\subsection{Projective removability}\label{Subsection:Remove}
On a plane curve determined by a nonconstant homogeneous polynomial, common zeros and poles of meromorphic
functions $f,g,h$ are removable after projectivization.
The next lemma generalizes \cite[Lem.~2.1]{GarciaFLT},
which treated the Fermat polynomial $P(x,y,z) = x^n+y^n-z^n$.  
No smoothness or irreducibility assumption on $\Gamma_P$ is needed.

\begin{lemma}\label{Lemma:Removable}
Let $P(x,y,z)$ be a nonconstant homogeneous polynomial, let $\Omega$ be a Riemann surface, and let $f,g,h$ 
be meromorphic functions on $\Omega$, not all identically zero, such that $P(f,g,h) = 0$.
Then $\Phi=[f:g:h]$, initially defined where $f,g,h$ are finite and not all zero, extends holomorphically to a map
$\Phi:\Omega \to \Gamma_P$.  In particular, any common zeros and poles of $f,g,h$ are projectively removable.
\end{lemma}

\begin{proof}
Let $n = \deg P$, fix $a\in\Omega$, and fix a local coordinate $w$ centered at $a$. 
With the convention $\ord_a 0=\infty$, let $m = \min\{ \ord_a f, \ord_a g, \ord_a h\}$.
Since $f,g,h$ are not all identically zero, $m$ is finite.  On a neighborhood of $a$, we have
\begin{equation}
f(w)=w^{m}f_0(w), \qquad
g(w)=w^{m}g_0(w), \quad \text{and} \quad
h(w)=w^{m}h_0(w),
\end{equation}
in which $f_0$, $g_0$, and $h_0$ are holomorphic near $a$ and at least one of
$f_0(a), g_0(a), h_0(a)$ is nonzero.  
Thus, $[f_0:g_0:h_0]$ is a holomorphic map into $\P^2$ on a neighborhood of $a$.
Since $P(x,y,z)$ is homogeneous of degree $n$, we get
$P(f,g,h) = w^{mn}P(f_0,g_0,h_0) = 0$
on a punctured neighborhood of $a$.
Hence $P(f_0,g_0,h_0) = 0$ there, and therefore at $a$ too.
Thus, $[f_0:g_0:h_0]$ maps into $\Gamma_P$.  On the punctured neighborhood of $a$, we have
$[f_0:g_0:h_0] = [f:g:h]$, which yields a holomorphic extension of the projective map across $a$.
Since $a \in \Omega$ is arbitrary and the local holomorphic extensions agree on overlaps, they
combine to give a holomorphic map $\Phi:\Omega\to \Gamma_P$.
\end{proof}

\section{Proof of Theorem~\ref{Theorem:Main}.a}\label{Section:ProofA}

The proof of Theorem~\ref{Theorem:Main}.a involves some material which may not be
familiar to all readers.  Consequently, we proceed in a deliberate manner since we hope that these techniques can be more broadly applied to similar problems in the future.

\subsection{The spherical derivative of a holomorphic curve}\label{Subsection:SphericalDerivative}

Let $\Omega\subset\C$ be a domain; that is, a nonempty, connected, open set.   Suppose that
$\Phi:\Omega\to\PP^2$ is holomorphic. Locally, we have $\Phi=[f:g:h]$, in which $f,g,h$ are holomorphic with no common zeros.  
The \emph{spherical derivative} of $\Phi$ is defined by
\begin{equation}\label{eq:SphericalDerivative}
\Phi^{\sharp} =  \bigg( \frac{ |fg'-gf'|^2 +  |fh'-hf'|^2 + |gh'-hg'|^2 }{ ( |f|^2+|g|^2+|h|^2 )^2 } \bigg)^{1/2};
\end{equation}
see \cite[Sec.~2, (2.2)--(2.5)]{CherryEremenko} for alternate formulas.
This is the two-dimensional case of the \emph{Fubini--Study derivative} \cite[p.~1397]{BarrettEremenko}.
If we replace $\PP^2$ with $\PP^1$ (the Riemann sphere) and let $\Phi=[1:f]$, then \eqref{eq:SphericalDerivative} becomes 
$\Phi^{\sharp} = |f'|/(1+|f|^2)$, the classical spherical derivative of $f$ \cite[p.~2]{EremenkoNormal}.
We use the ``sharp'' superscript ${}^{\sharp}$ here instead of ${}^{\#}$, which we used in \eqref{eq:Reflection} to denote
holomorphic reflection through $\T$.  

Observe that $\Phi^{\sharp}$ depends only on the projective map $\Phi$, not on its representation;
this is mentioned in passing below \cite[eq.~(3)]{EremenkoNormal}, although in our case it is easily verified.
That is, $[f:g:h]^{\sharp} = [ \phi f : \phi g : \phi h]^{\sharp}$ when $\phi$ is nonvanishing and holomorphic.  
If we replace $(f,g,h)$ with $(\phi f, \phi g, \phi h)$ in \eqref{eq:SphericalDerivative}, then 
$|fg'-gf'|^2$ becomes $| \phi|^4 | fg'-gf'|^2$ and similarly for the other two terms in the numerator.
The denominator scales by $|\phi|^4$, which cancels with the same factor from the numerator.
In particular, when a projective map is initially represented by meromorphic coordinates having common zeros or
poles, these may first be removed as in Lemma~\ref{Lemma:Removable}
without changing $\Phi^{\sharp}$.

If $U\subset\C$ is a domain and $\psi:U\to\Omega$ is holomorphic, then we have the chain rule:
\begin{equation}\label{eq:SphericalChainRule}
(\Phi\circ\psi)^{\sharp} = |\psi'| (\Phi^{\sharp}\circ\psi) 
\end{equation}
since
\begin{equation*}
(f\circ\psi)(g\circ\psi)'  -(g\circ\psi)(f\circ\psi)' = \psi' \big( (fg'-gf')\circ\psi\big),
\end{equation*}
and similarly for the other two terms from the numerator of \eqref{eq:SphericalDerivative}.  Substitute these results 
into \eqref{eq:SphericalDerivative} to obtain the desired formula \eqref{eq:SphericalChainRule}.

\begin{lemma}\label{Lemma:SphericalCritical}
Let $\Omega\subset\C$ be a domain and let $\Phi:\Omega\to\PP^2$ be holomorphic.  Then
$\Phi^{\sharp}(z_0)=0$ if and only if $\Phi'(z_0)=0$.
\end{lemma}

\begin{proof}
Write $\Phi=[f:g:h]$ near $z_0$, in which $f,g,h$ are holomorphic with no common zeros.  
Without loss of generality, suppose that $f(z_0) \neq 0$. On a neighborhood of $z_0$, let
$u = g/f$ and $v = h/f$.
In the corresponding affine chart of $\PP^2$, the map $\Phi$ is represented by
$z\mapsto (u(z),v(z))$. Moreover,
\begin{equation}
fg'-gf'=f^2u', \qquad
fh'-hf'=f^2v', \quad\text{and} \quad
gh'-hg'=f^2(uv'-vu').
\end{equation}
Substitute these identities into \eqref{eq:SphericalDerivative} and obtain
\begin{equation}
(\Phi^{\sharp})^2 = \frac{|u'|^2+|v'|^2+|uv'-vu'|^2}{(1+|u|^2+|v|^2)^2}.
\end{equation}
The denominator is positive, so $\Phi^{\sharp}(z_0)=0$ if and only if $u'(z_0)=v'(z_0)=0$.
But this is precisely the condition that the derivative of the local coordinate representation
$z\mapsto (u(z),v(z))$ vanishes at $z_0$.  Thus, $\Phi'(z_0) = 0$ if and only if $\Phi^{\sharp}(z_0)=0$.
\end{proof}

\subsection{Potential theory}\label{Subsection:Potential}
Let $L^1_{\mathrm{loc}}(\Omega)$ denote the space of locally integrable functions on $\Omega$ with respect to
planar Lebesgue measure $dA(z)=dx\,dy$ and let $C_{\mathrm{c}}^\infty(\Omega)$ denote the space of infinitely differentiable
functions on $\Omega$ with compact support.
Each $u\in L^1_{\mathrm{loc}}(\Omega)$ defines a distribution by
\begin{equation*}
\langle u,\phi\rangle = \int_\Omega u(z)\phi(z)\,dA(z)
\end{equation*}
in which $\phi\in C_{\mathrm{c}}^\infty(\Omega)$.
Its \emph{distributional Laplacian} is the distribution $\Delta u$ defined by
\begin{equation*}
\langle\Delta u,\phi\rangle
=
\int_\Omega u(z)\Delta\phi(z)\,dA(z),
\end{equation*}
in which $\phi\in C_{\mathrm{c}}^\infty(\Omega)$ and
\begin{equation}\label{eq:Delta}
\Delta =
\frac{\partial^2}{\partial x^2} + \frac{\partial^2}{\partial y^2}
= 4\frac{\partial^2}{\partial z\,\partial\overline{z}}.
\end{equation}
If $u$ is subharmonic and not identically $-\infty$, then $u\in L^1_{\mathrm{loc}}(\Omega)$ \cite[Thm.~2.5.1]{Ransford}
and its distributional Laplacian is a positive Radon measure \cite[Sec.~3.7]{Ransford}; that is,
$\langle \Delta u,\phi\rangle \geq 0$ for every nonnegative $\phi\in C_{\mathrm{c}}^\infty(\Omega)$.
The corresponding \emph{Riesz measure} is
\begin{equation*}
\mu_u=\frac{1}{2\pi}\Delta u.
\end{equation*}
This normalization ensures that $\Delta\log|z-a|=2\pi\delta_a$
in the sense of distributions. Thus, the Riesz measure of $\log|z-a|$ is the point mass $\delta_a$ \cite[Thm.~3.7.4]{Ransford}.

\begin{lemma}\label{Lemma:RieszSpherical}
Let $\Phi:\Omega\to\PP^2$ be holomorphic and suppose that $\Phi=[f:g:h]$
wherever the projective triple is defined, in which $f,g,h$ are holomorphic
functions on $\Omega$, not all identically zero.  Then
\begin{equation}\label{eq:LogFGH}
u = \frac{1}{2}\log\big( |f|^2+|g|^2+|h|^2 \big)
\end{equation}
is subharmonic and not identically $-\infty$.  Moreover,
\begin{equation}\label{eq:RieszSharp}
d\mu_u = \frac{1}{\pi}(\Phi^{\sharp})^2\,dA + \sum_a m_a\delta_a,
\end{equation}
in which the sum runs over the common zeros of $f,g,h$ and
\begin{equation}
m_a=\min\{\ord_a f,\ord_a g,\ord_a h\}.
\end{equation}
In particular,
\begin{equation}\label{eq:RieszDominatesSharp}
\frac{1}{\pi}(\Phi^{\sharp})^2\,dA\leq d\mu_u.
\end{equation}
\end{lemma}

\begin{proof}
Since $f,g,h$ are not all identically zero, $u$ is not identically $-\infty$. Fix $a\in\Omega$ and let $w$ be a local coordinate centered at $a$ and define
$m=\min\{\ord_a f,\ord_a g,\ord_a h\}$, which is nonnegative. On a neighborhood of $a$, write $f=w^m f_0$, $g=w^m g_0$, and $h=w^m h_0$, in which $f_0,g_0,h_0$ are holomorphic and have no common zero near $a$.  Thus, $\Phi=[f_0:g_0:h_0]$ on this neighborhood. Let
\begin{equation}
u_0 = \frac{1}{2}\log(|f_0|^2+|g_0|^2+|h_0|^2 ),
\end{equation}
which is smooth since $f_0,g_0,h_0$ have no common zeros.
Then
\begin{equation}\label{eq:TakeLaplace}
u=m\log|w|+u_0,
\end{equation}
so \eqref{eq:Delta} and \cite[eq.~(2.2)]{CherryEremenko} (or a straightforward computation) give
\begin{equation}\label{eq:DeltaSphere}
\Delta u_0
= 4\frac{\partial^2 u_0}{\partial z\,\partial\overline z}\\
= 2\frac{\partial^2}{\partial z\,\partial\overline z} \log(|f_0|^2+|g_0|^2+|h_0|^2 )\\
= 2(\Phi^{\sharp})^2.
\end{equation}
In particular, $\Delta u_0\geq0$, so $u_0$ is subharmonic \cite[Thm.~2.4.4]{Ransford}.  Since $\log|w|$ is subharmonic,
\eqref{eq:TakeLaplace} shows that $u$ is subharmonic.  This also follows from \cite[Prop.~5.2]{Berg}.

In the sense of distributions,
\begin{equation}\label{eq:PointMasses}
\frac{1}{2\pi}\Delta\log|w|=\delta_a;
\end{equation}
see \cite[Thm.~3.7.8, p.~76]{Ransford}.  Apply the distributional
Laplacian to \eqref{eq:TakeLaplace}, use \eqref{eq:DeltaSphere} and \eqref{eq:PointMasses}, and deduce that
the following holds locally
\begin{equation}
d\mu_u = \frac{1}{\pi}(\Phi^{\sharp})^2\,dA + m \delta_a.
\end{equation}
The common zeros of $f,g,h$ are discrete in $\Omega$ since at least one of these functions is not identically zero.  Thus, the 
local identities above combine to give
\begin{equation}
d\mu_u =\frac{1}{\pi}(\Phi^{\sharp})^2\,dA + \sum_a m_a\delta_a,
\end{equation}
in which the sum is locally finite.  This is \eqref{eq:RieszSharp}, from which
\eqref{eq:RieszDominatesSharp} follows.
\end{proof}

\begin{remark}\label{Remark:LebesgueDecomposition}
The formula \eqref{eq:RieszSharp} is the Lebesgue decomposition of the Riesz measure with respect to area measure: its absolutely continuous part is $\frac{1}{\pi}(\Phi^{\sharp})^2\,dA$ and
its singular part is $\sum_a m_a\delta_a$.  Thus, \eqref{eq:RieszDominatesSharp} amounts to discarding the singular part.
\end{remark}

\subsection{Nevanlinna representations}\label{Subsection:Nevanlinna}
Let $\Omega \subseteq \C$ be a domain; that is, a nonempty, connected, open set.
Recall that a holomorphic function $f$ on $\Omega$ belongs to the \emph{Nevanlinna class} $\N(\Omega)$
if the subharmonic function $\log^+|f|$ has a harmonic majorant on $\Omega$ \cite[p.~66]{Garnett};
here $\log^+ x = \max\{\log x,0\}$ (the subharmonicity of $\log^+|f|$ is \cite[Thm.~17.3]{Rudin}).
The cases $\Omega = \D$,  $\Omega = \Cp$ (upper half-plane), and $\Omega = \Cm$ (lower half-plane) are well studied and appear in the standard sources \cite{Duren, Garnett, Koosis, Nikolski, Conway, Rudin}.

The Nevanlinna class is conformally invariant \cite[p.~66]{Garnett}: if $\phi:\Omega_1\to\Omega_2$ is a conformal map, then
$f\in\N(\Omega_2)$ if and only if $f\circ\varphi\in\N(\Omega_1)$  since harmonic majorants pull back under conformal maps.
The analogous statement holds for anticonformal maps: if $\varphi:\Omega_1\to\Omega_2$ is anticonformal and $f\in\N(\Omega_2)$, then the holomorphic map $z\mapsto \overline{f(\varphi(z))}$ belongs to $\N(\Omega_1)$.

A \emph{projective lift} of a holomorphic map $\Phi:\Omega\to\PP^2$ is a holomorphic triple $(f,g,h)$ with no common zeros
such that $\Phi=[f:g:h]$ on $\Omega$.  Such lifts are not unique, since multiplication of all three
components by the same nowhere-vanishing holomorphic function leaves the
projective map unchanged.

We also use the weaker notion of a \emph{projective representation}.  A holomorphic projective representation is a holomorphic triple $(f,g,h)$,
not all identically zero, which represents $\Phi$ wherever the projective triple is defined.  Common zeros are locally removable after
projectivization by the same factorization argument used in the proof of Lemma~\ref{Lemma:Removable}.  
A \emph{meromorphic projective representation} permits meromorphic components; common zeros and poles are removed
projectively in the same manner.  This distinction is useful since triples that arise from pseudocontinuation need not initially be lifts: they
may have common zeros or poles even though the associated projective map is holomorphic.

\begin{lemma}\label{Lemma:Nevanlinna}
Let $P(x,y,z)$ be a nonconstant homogeneous polynomial, let $\theta$ be an inner function with $\#(\sigma(\theta)\cap \T)\leq 1$, and suppose that
$f,g,h\in K_{\theta}^p$ are not all identically zero and satisfy $P(f,g,h)= 0$.
Then there is a holomorphic map $\Phi:\C\to \Gamma_P$
whose restrictions to $\Cp$ and $\Cm$ admit projective representations whose coordinates belong to $\N(\Cp)$ and $\N(\Cm)$, respectively.
\end{lemma}

\begin{proof}
Since $E=\sigma(\theta)\cap\T$ contains at most one point, choose $\zeta\in\T$ such that $E\subseteq\{\zeta\}$.
For $f,g,h\in K_{\theta}^p$, let $\widetilde{f} = \overline{fz}\theta$, $\widetilde{g} = \overline{gz}\theta$, and $\widetilde{h} =\overline{hz}\theta$, respectively, denote their conjugates in $K_{\theta}^p$. Lemma~\ref{Lemma:Removable} ensures that the projective triple
$\Phi_{\mathrm{in}}=[f:g:h]$ extends to a holomorphic map $\Phi_{\mathrm{in}}:\D\to\Gamma_P$.

On $\De$, consider
\begin{equation*}
\Phi_{\mathrm{out}} = [\widetilde{f}^{\#}: \widetilde{g}^{\#}: \widetilde{h}^{\#}];
\end{equation*}
the three coordinates are not all identically zero.  Then \eqref{eq:ExteriorPolynomialRelation} ensures that
\begin{equation*}
P(\widetilde{f}^{\#},\widetilde{g}^{\#},\widetilde{h}^{\#})=0
\end{equation*}
on $\De$, so Lemma~\ref{Lemma:Removable} gives a holomorphic map $\Phi_{\mathrm{out}}:\De\to\Gamma_P$.
Every point of $\T\setminus\{\zeta\}$ lies on an open arc $I$ disjoint from $E$.  By \eqref{eq:PseudocontinuationLocal}, the projective triples defining
$\Phi_{\mathrm{in}}$ and $\Phi_{\mathrm{out}}$ agree wherever they are defined on a neighborhood of $I$.  
Since both projective maps extend holomorphically across the common zeros of their defining triples, they agree on the neighborhood.
Thus, $\Phi_{\mathrm{in}}$ and $\Phi_{\mathrm{out}}$ glue to a holomorphic map $\Psi:\widehat{\C}\setminus\{\zeta\}\to\Gamma_P$. If $E=\varnothing$, they glue across all of $\T$ and hence $\zeta$ is irrelevant.

The M\"obius transformation
\begin{equation*}
\kappa(z) = \zeta\frac{z-i}{z+i}.
\end{equation*}
is a biholomorphic map
$\kappa: \C\to \widehat{\C}\setminus\{\zeta\}$ 
with $\kappa(\Cp)=\D$ and $\kappa(\Cm)=\De$.
Thus, $\Phi=\Psi\circ\kappa$ is a holomorphic map $\Phi:\C\to\Gamma_P$.
On $\Cp$, the map $\Phi$ has the projective representation
$[f\circ\kappa: g\circ\kappa: h\circ\kappa]$. Since $\kappa:\Cp\to\D$ is conformal and
$f,g,h\in K_{\theta}^p\subseteq H^p\subseteq\N(\D)$, conformal invariance yields $f\circ\kappa,\;g\circ\kappa,\;h\circ\kappa\in\N(\Cp)$.

On $\Cm$, the map $\Phi$ has the projective representation
\begin{equation}\label{eq:LowerRepresentation}
[\widetilde{f}^{\#}\circ\kappa: \widetilde{g}^{\#}\circ\kappa: \widetilde{h}^{\#}\circ\kappa].
\end{equation}
Let $\rho(z)= 1 / \overline{\kappa(z)}$. Then $\rho:\Cm\to\D$ is anticonformal and
\begin{equation*}
\widetilde{f}^{\#}(\kappa(z)) = \overline{\widetilde{f}(\rho(z))}, \qquad
\widetilde{g}^{\#}(\kappa(z)) = \overline{\widetilde{g}(\rho(z))}, \quad\text{and}\quad
\widetilde{h}^{\#}(\kappa(z)) = \overline{\widetilde{h}(\rho(z))}.
\end{equation*}
Since $\widetilde{f},\widetilde{g},\widetilde{h}\in K_\theta^p\subseteq H^p\subseteq\N(\D)$, anticonformal invariance yields
\begin{equation*}
\widetilde{f}^{\#}\circ\kappa,\, \widetilde{g}^{\#}\circ\kappa,\, \widetilde{h}^{\#}\circ\kappa \in\N(\Cm). \qedhere
\end{equation*}
\end{proof}

\begin{remark}
There is no smoothness hypothesis on $\Gamma_P$ in the previous lemma.  
Smoothness is only required in Lemma \ref{Lemma:Elliptic} below.
\end{remark}

\subsection{A growth estimate}\label{Subsection:Growth}
The next lemma uses the Nevanlinna-class hypothesis to control a weighted
area integral involving the spherical derivative of a holomorphic map $\Phi:\C\to\PP^2$.

\begin{lemma}\label{Lemma:Growth}
Suppose that $\Phi:\C\to\PP^2$ is holomorphic and that $\Phi|_{\Cp}$ and $\Phi|_{\Cm}$ admit projective representations whose coordinates belong to $\N(\Cp)$ and $\N(\Cm)$, respectively.  Then there exists an $R>0$ such that
\begin{equation}\label{eq:WeightedSharp}
\int_{|z|>R} \frac{|\Im z|}{|z|^2} (\Phi^{\sharp}(z))^2\,dA(z) < \infty.
\end{equation}
\end{lemma}

\begin{proof}
We prove the result for $\Cp$; the argument in $\Cm$ is similar.
Let $\Phi=[f:g:h]$ on $\Cp$, in which $f,g,h\in \N(\Cp)$ are not all identically zero.
Choose harmonic majorants $H_f,H_g,H_h$ of $\log^+|f|,\log^+|g|,\log^+|h|$, respectively \cite[Thm.~5.1]{Garnett}.
Lemma \ref{Lemma:RieszSpherical} ensures that $u=\frac{1}{2}\log(|f|^2+|g|^2+|h|^2)$ is subharmonic and not identically $-\infty$.
It has a harmonic majorant on $\Cp$ since
\begin{equation*}
u \leq \frac{1}{2} \log 3+ \log^+ |f|+\log^+ |g|+\log^+ |h| \leq \frac{1}{2}\log 3 + H_f+H_g+H_h.
\end{equation*}
Let $H$ denote the least harmonic majorant of $u$  \cite[Def.~4.5.3, Th.~4.5.4]{Ransford}
and let $\mu_u=(2\pi)^{-1}\Delta u$ denote the Riesz measure for $u$.

Choose $a=\alpha+i\beta\in\Cp$ with $u(a)>-\infty$ and $\beta>0$. Then let
\begin{equation}\label{eq:Rplus}
R_+\geq \max\{2|a|,16\beta\}.
\end{equation}
Recall that the Green function of $\Cp$ is
\begin{equation}
G_+(z,a) = \log\bigg|\frac{z-\overline a}{z-a}\bigg|;
\end{equation}
see \cite[Table~4.2]{Ransford}.  If $z=x+iy\in\Cp$ and $|z|\geq R_+$, then
\begin{align}
G_+(z,a)
&= \frac{1}{2} \log \frac{(x-\alpha)^2+(y+\beta)^2}{(x-\alpha)^2+(y-\beta)^2}
= \frac{1}{2}\log\bigg( 1+\frac{4\beta y}{|z-a|^2} \bigg) \nonumber\\
&\geq \frac{\beta y}{|z-a|^2} \nonumber\\
&\geq \frac{\beta y}{4|z|^2} \label{eq:GreenLowerBound}
\end{align}
since \eqref{eq:Rplus} ensures that
\begin{equation}
|z-a|\geq |z|-|a|\geq\frac{|z|}{2}
\qquad\text{and}\qquad
|z-a|\leq |z|+|a| \leq 2|z|,
\end{equation}
and hence
\begin{equation}
0 \leq \frac{4\beta y}{|z-a|^2} \leq \frac{16\beta}{|z|} \leq 1,
\end{equation}
which permits the use of the inequality $\log(1+t)\geq t/2$ for $0\leq t\leq1$.

Since $H$ is the least harmonic majorant of $u$,
\begin{equation}\label{eq:RieszGreen}
H(a) - u(a) = \int_{\Cp}G_+(z,a)\,d\mu_u(z) < \infty
\end{equation}
by \cite[Thm.~4.5.4]{Ransford}. Then Lemma~\ref{Lemma:RieszSpherical} ensures that
\begin{equation}
\frac{1}{\pi}(\Phi^{\sharp})^2\,dA
\leq
d\mu_u.
\end{equation}
Combining this with \eqref{eq:GreenLowerBound} and
\eqref{eq:RieszGreen} yields
\begin{align}
\int_{\substack{z\in\Cp\\|z|>R_+}} \frac{\Im z}{|z|^2}(\Phi^{\sharp}(z))^2\,dA(z)
&\leq \frac{4}{\beta} \int_{\Cp} G_+(z,a)(\Phi^{\sharp}(z))^2\,dA(z) \nonumber\\
&\leq \frac{4\pi}{\beta} \int_{\Cp}G_+(z,a)\,d\mu_u(z) \nonumber\\
&< \infty. \label{eq:UpperGrowth}
\end{align}

Applying the same argument on $\Cm$ provides an $R_->0$ such that
\begin{equation}
\int_{\substack{z\in\Cm\\|z|>R_-}} \frac{|\Im z|}{|z|^2}(\Phi^{\sharp}(z))^2\,dA(z) < \infty.
\end{equation}
Let $R=\max\{R_+,R_-\}$, combine the two estimates above, and obtain the desired inequality
\eqref{eq:WeightedSharp} since $\R$ has zero area measure.
\end{proof}

\subsection{The genus-one hypothesis}\label{Subsection:GenusOne}
The proof of the final lemma needed for Theorem~\ref{Theorem:Main}.a is the first time the hypothesis $n=3$ enters. 
By \eqref{eq:Genus}, the smooth plane cubic $\Gamma_P$ has genus $1$. 
We exploit the fact that smooth projective genus-$1$ curves become elliptic curves (after choosing a base point).

\begin{lemma}\label{Lemma:Elliptic}
Let $\Gamma\subset\PP^2$ be a smooth projective cubic curve and let $\Phi:\C\to\Gamma$ be holomorphic.  Suppose that the restrictions of
$\Phi$ to $\Cp$ and $\Cm$ admit projective representations whose coordinates belong to $\N(\Cp)$ and $\N(\Cm)$, respectively.  
Then $\Phi$ is constant.
\end{lemma}

\begin{proof}
Since $\Gamma$ is a smooth projective plane cubic, \eqref{eq:Genus} ensures that it has genus $1$.
The uniformization theorem provides a lattice $\Lambda\subset\C$ and universal covering map
$\pi:\C\to\Gamma$ such that $\Gamma$ is biholomorphic to $\C/\Lambda$ and
\begin{equation}\label{eq:LatticePeriodicity}
\pi(w+\lambda)=\pi(w)
\end{equation}
for all $w\in\C$ and $\lambda \in \Lambda$ \cite[Sec.~2.7 and Thm.~4.4.1]{Jost}, \cite[Ch.~VI, Sec.~5]{Silverman}, \cite[Ch.~VII, Sec.~3]{Miranda}, \cite[Ch.~III]{Farkas}. Since $\C$ is simply connected, $\Phi$ lifts through $\pi$.  Thus, there is an entire function $k$ such that
\begin{equation}\label{eq:PhiEllipticFactorization}
\Phi=\pi\circ k.
\end{equation}

The spherical derivative $\pi^{\sharp}$, defined by \eqref{eq:SphericalDerivative}, is nonnegative and continuous.
Moreover, \eqref{eq:LatticePeriodicity} and the chain rule \eqref{eq:SphericalChainRule} ensure that
$\pi^{\sharp}(w+\lambda)=\pi^{\sharp}(w)$ for all $w\in\C$ and $\lambda \in \Lambda$.
We claim that $\pi^{\sharp}(w) > 0$ for all $w\in\C$.
Indeed, fix $w_0\in\C$.  Since $\pi$ is a covering map, it is biholomorphic on a neighborhood $U$ of $w_0$.  
In local coordinates,
\begin{equation}
((\pi|_U)^{-1})'(\pi(w_0)) \pi'(w_0)=1,
\end{equation}
so $\pi'(w_0)\neq0$.  Then Lemma~\ref{Lemma:SphericalCritical} ensures that $\pi^{\sharp}(w_0)\neq 0$, so
$\pi^{\sharp}(w_0)>0$.
Continuity and periodicity ensure that $(\pi^{\sharp})^2$
attains a positive minimum on the closure of any fundamental parallelogram for $\Lambda$.  
Thus, there is a  $c>0$ such that
\begin{equation}\label{eq:PiSharpPositiveLowerBound}
(\pi^{\sharp}(w))^2 \geq c
\end{equation}
for all $w\in\C$. Then \eqref{eq:SphericalChainRule}, \eqref{eq:PhiEllipticFactorization}, and \eqref{eq:PiSharpPositiveLowerBound} yield
\begin{equation}\label{eq:PhiSharpDominatesKPrime}
(\Phi^{\sharp}(z))^2 = (\pi^{\sharp}(k(z)))^2|k'(z)|^2 \geq c|k'(z)|^2.
\end{equation}

Lemma~\ref{Lemma:Growth} provides an $R>0$ such that
\begin{equation}\label{eq:WeightedKPrime}
\int_{|z|>R} \frac{|\Im z|}{|z|^2} |k'(z)|^2\,dA(z)
\leq c^{-1} \int_{|z|>R} \frac{|\Im z|}{|z|^2} (\Phi^{\sharp}(z))^2\,dA(z) <\infty.
\end{equation}
Write $z=re^{it}$ so that $|\Im z|=r|\sin t|$ and $dA(z)=r\,dr\,dt$.
Then the left side of the inequality \eqref{eq:WeightedKPrime} becomes
\begin{equation}\label{eq:WeightedKPrimePolar}
\int_R^\infty\int_0^{2\pi} |k'(re^{it})|^2|\sin t|\,dt\,dr <\infty.
\end{equation}
We claim that $k' \equiv 0$.  Suppose toward a contradiction that
\begin{equation*}
k'(z) = a_mz^m+a_{m+1}z^{m+1}+\cdots
\end{equation*}
with $a_m \neq 0$.  Since $m$ is minimal and $|\sin t| \geq \sin^2t = \frac{1}{4}|1-e^{2it}|^2$,
the Fourier coefficient of frequency $m$ in $(1-e^{2it})k'(re^{it})$ is $a_mr^m$, so
Parseval's formula yields
\begin{equation}\label{eq:ParsevalKPrimeLowerBound}
\int_0^{2\pi}|k'(re^{it})|^2|\sin t|\,dt \geq \frac{\pi}{2}|a_m|^2r^{2m},
\end{equation}
which is not integrable on $[R,\infty)$.  Since this contradicts
\eqref{eq:WeightedKPrimePolar}, we conclude that $k' \equiv 0$, so $k$ is constant.
It follows from \eqref{eq:PhiEllipticFactorization} that $\Phi$ is constant.
\end{proof}

\subsection{Completion of the proof of Theorem~\ref{Theorem:Main}.a}\label{Section:Finish}
Suppose that $n=3$ and $\#(\sigma(\theta)\cap\T)\leq 1$.  If $f \equiv g \equiv h \equiv 0$, there is nothing to prove.  
Otherwise, Lemma~\ref{Lemma:Nevanlinna} provides a holomorphic map
$\Phi:\C\to\Gamma_P$ whose restrictions to $\Cp$ and $\Cm$ admit
projective representations with Nevanlinna-class coordinates.  Since
$\Gamma_P$ is a smooth cubic, Lemma~\ref{Lemma:Elliptic} ensures that $\Phi$ is constant.  
The construction in Lemma \ref{Lemma:Nevanlinna} and the fact that $\kappa(\C_+)=\D$ provide  an $[a:b:c]\in\Gamma_P$ such that
\begin{equation*}
[f(z):g(z):h(z)]=[a:b:c]
\end{equation*}
whenever $f(z)$, $g(z)$, and $h(z)$ are not all zero.
At least one of $a,b,c$ is nonzero.  If $a\neq0$, then $ag \equiv bf$ and $ah \equiv cf$ on $\D$ since the two sides
agree wherever the projective triple is defined and their respective differences are holomorphic. Thus,
$g \equiv \frac{b}{a}f$ and $h \equiv \frac{c}{a}f$ on $\D$. The cases $b\neq0$ and $c\neq0$ are similar. Consequently, $\dim\Span\{f,g,h\}\leq1$. This completes the proof. \qed

\section{Proof of Theorem~\ref{Theorem:Main}.b}\label{Section:ProofB}
The proof generalizes the method of \cite{GarciaFLT}.
We require a result of J\"arvi \cite[Cor.~1]{Jarvi}; see also previous related work of Nishino
\cite[Thm.~II, pp.~110--111]{Nishino} and Suzuki \cite{Suzuki}.

\begin{lemma}\label{Lemma:Extension}
Let $\Omega\subset\C$ be a domain, let $E\subset \Omega$ be compact with
$\caplog E=0$, and let $X$ be a compact Riemann surface of genus at least two.  Every holomorphic map
$\Phi:\Omega\setminus E\to  X$ extends holomorphically to $\Omega$.
\end{lemma}

We also require the following standard hyperbolic rigidity theorem \cite[Thm.~2.6.3]{Jost}.

\begin{lemma}\label{Lemma:HyperbolicRigidity}
Let $X$ be a compact Riemann surface of genus at least two.  Every holomorphic map
$\Phi:\C\to  X$ is constant.
\end{lemma}

The proof of the lemma is short enough to state here. The universal covering surface of $X$ is $\D$.  Since $\C$ is simply connected, $\Phi$ lifts to a holomorphic map $\widetilde{\Phi}:\C\to \D$, so Liouville's theorem implies that $\widetilde{\Phi}$ is constant.  Thus, $\Phi$ is constant.

\begin{proof}[Proof of Theorem~\ref{Theorem:Main}.b]
Let $E=\sigma(\theta)\cap\T$, so $\caplog E=0$.  If $f \equiv g \equiv h \equiv 0$, there is nothing to prove.  
Otherwise, Lemma~\ref{Lemma:Removable}, together with
\eqref{eq:PseudocontinuationLocal} and \eqref{eq:ExteriorPolynomialRelation}, provides holomorphic maps
\begin{equation}
[f:g:h]:\D\to\Gamma_P \qquad\text{and}\qquad
[\widetilde{f}^{\#}:\widetilde{g}^{\#}:\widetilde{h}^{\#}]:\De\to\Gamma_P
\end{equation}
whose restrictions agree wherever these projective triples are defined.  Thus, the extended maps agree across every open arc in $\T\setminus E$
and hence they glue to a holomorphic map $\Psi:\widehat{\C}\setminus E\to\Gamma_P$.
Restrict $\Psi$ to $\C\setminus E$ and apply
Lemma~\ref{Lemma:Extension} with $\Omega=\C$ to obtain a holomorphic extension
$\widetilde{\Psi}:\C\to\Gamma_P$.
Since $\deg P = n\geq4$ and $\Gamma_P$ is smooth, \eqref{eq:Genus} gives $g(\Gamma_P) \geq3$.
Lemma~\ref{Lemma:HyperbolicRigidity} ensures that $\widetilde{\Psi}$ is constant.  Thus, $[f:g:h]$ is constant on $\D$,
so there is an $[a:b:c]\in\Gamma_P$ such that $[f:g:h]=[a:b:c]$ wherever the projective triple is defined.  As in the proof of
Theorem \ref{Theorem:Main}.a, we conclude that $\dim\Span\{f,g,h\}\leq1$.
\end{proof}

\bibliography{FLTSIS2}
\bibliographystyle{amsplain}

\end{document}